\documentclass[pdflatex,sn-mathphys-num]{sn-jnl}

\usepackage[pagewise]{lineno}

\usepackage[utf8]{inputenc}
\usepackage[T1]{fontenc}

\usepackage{graphicx}
\usepackage{amsmath,amssymb,amsfonts}
\usepackage{amsthm}
\usepackage{mathrsfs}
\usepackage{xcolor}
\usepackage{textcomp}
\usepackage{booktabs}
\usepackage{float}
\usepackage{caption}
\usepackage{url}

\usepackage[inoutnumbered,ruled]{algorithm2e}

\theoremstyle{thmstyleone}%
\newtheorem{theorem}{Theorem}
\newtheorem{proposition}[theorem]{Proposition}%
\newtheorem{lemma}[theorem]{Lemma}
\newtheorem{corollary}[theorem]{Corollary}

\theoremstyle{thmstyletwo}%
\newtheorem{assumption}[theorem]{Assumption}
\newtheorem{remark}{Remark}%

\theoremstyle{thmstylethree}%
\newtheorem{definition}{Definition}%

\DeclareMathOperator{\Real}{Re}

\usepackage[figcolor=white]{todonotes}

\begin{document}

\title[Transparent BCs for the Time-Fractional Heat Equation on Metric Graphs]
{Transparent Boundary Conditions for the Time-Fractional \\ Heat Equation on Metric Graphs}

\author*[1,2]{\fnm{Anis} \sur{Bel Hadj Hassin}}\email{anis.hadjhassin@gmail.com}

\author[3]{\fnm{Matthias} \sur{Ehrhardt}}\email{ehrhardt@uni-wuppertal.de}

\affil*[1]{\orgdiv{LAMSIN}, \orgname{University of Tunis El Manar, National Engineering School of Tunis}, \orgaddress{\postcode{1002}, \city{Tunis}, \country{Tunisia}}}

\affil[2]{\orgname{ESPRIT School of Engineering}, \orgaddress{\street{1, 2 rue Andr\'e Amp\`ere}, \postcode{2083}, \city{Pôle Technologique}, \country{Tunisia}}}

\affil[3]{\orgdiv{Applied and Computational Mathematics}, \orgname{University of Wuppertal}, \orgaddress{\street{Gau\ss{}strasse 20}, \postcode{42119}, \city{Wuppertal}, \country{Germany}}}

\abstract{In this work, we study transparent boundary conditions for the time-fractional heat equation on a metric star graph. 
This equation is the standard continuum model for subdiffusive heat and mass transport in disordered and porous branched media. 
Using the Laplace transform to solve the exterior branch problems, we demonstrate that the non-local vertex operator observed by the incoming bond is an Abel-type convolution kernel of order $\beta/2$, which generalizes the classical half-derivative kernel recovered at $\beta=1$.  
Our main result is that the diffusion-coefficient sum rule $\kappa_1=\kappa_2+\kappa_3$, known to eliminate thermal backflow for the classical heat equation, remains unchanged for every fractional order $\beta$: the transparency of the junction depends only on the network's diffusivities and not on the memory of the transport law. 
We rigorously prove this and uniqueness for the coupled vertex problem. 
We confirm our results numerically at the level of the underlying coupled partial differential equation system with an implicit L1 finite-difference scheme. 
We compare the star-graph solution to a reference solution on an unbranched line with matching diffusivity.}



\maketitle

\section{Introduction}\label{sec:intro}
The design of tunable heat transport in nanostructured and disordered materials porous media, amorphous semiconductors, granular packings, heterogeneous polymer networks routinely runs into a common difficulty: 
local disorder, trapping sites, and a broad distribution of relaxation times cause the mean-square spread of thermal energy to grow sub-linearly in time, rather than linearly as predicted by the classical Fourier law. 
This anomalous, sub-diffusive behaviour is captured at the continuum level by replacing the ordinary first-order time derivative in the heat equation with a Caputo fractional derivative of order $\beta\in(0,1)$; 
the resulting time-fractional heat equation is now a standard modeling tool for thermal and mass transport in such media, with documented relevance to thermal management in nanostructured electronic and battery-electrode materials, heat and solute transport in porous rock and soil, and diffusive transport in biological tissue, where trapping and structural heterogeneity produce the same sub-linear spreading regardless of the physical carrier (heat, charge, or mass).

Many of the systems in which this anomalous transport is observed, moreover, are not simple one-dimensional conductors but branched networks porous rock, vascular and root systems, nanowire and nanotube junctions, branched polymers naturally modelled as metric (quantum) graphs: 
collections of one-dimensional bonds joined at vertices, with a differential operator on each bond coupled through vertex conditions~\cite{Kuchment2004, KostrykinSchrader1999}. 
Controlling how thermal (or, more generally, diffusive) energy splits at such a junction, without spurious reflection, is not merely a numerical convenience for truncating an otherwise unbounded simulation domain: it is the same design question that governs thermal management at a branched nanowire junction, solute exchange at a vascular or root bifurcation, or signal integrity at a branching interconnect, wherever the underlying transport is sub-diffusive rather than classical. 
In a recent work, Matrasulov, Yusupov and Ehrhardt~\cite{Matrasulov2026} showed how to construct a \textit{transparent boundary condition} (TBC) at such a vertex for the \textit{classical} heat equation. 
Combining the continuity of temperature and a Kirchhoff-type flux balance at a three-bond star graph with the Laplace-transform argument classically used to build TBCs on the real line~\cite{Ehrhardt2001}, they obtained a non-local-in-time boundary operator on the incoming bond whose convolution kernel is the classical Abel kernel $(t-\tau)^{-1/2}$. 
This vertex operator was shown to coincide exactly with the free-space (unbranched) transparent condition precisely when the diffusion coefficients satisfy the \textit{sum rule}
\begin{equation}\label{eq:sumrule}
  \kappa_1 = \kappa_2 + \kappa_3, 
\end{equation}
eliminating thermal backflow at the junction entirely.

The kernel $(t-\tau)^{-1/2}$ that emerges from this construction is, once recognized, nothing but a half-order Riemann-Liouville derivative in disguise precisely the operator that a Caputo time-derivative of order $1/2$ would produce.
Together with the physical relevance of sub-diffusive transport on branched networks recalled above, this observation raises the question we address here: what does a transparent junction look like once the transport law itself is fractional? 
We show below that solving the same exterior-branch problem for the time-fractional heat equation, with Caputo order $\beta\in(0,1)$ in place of the ordinary time derivative, replaces the half-order kernel by one of order $\beta/2$. 
What is not evident \textit{a priori} is whether the sum rule~\eqref{eq:sumrule}, derived for $\beta=1$, survives this generalization unchanged, or is instead replaced by some $\beta$-dependent condition. We show that it survives unchanged: relation~\eqref{eq:sumrule} eliminates backflow for every $\beta\in(0,1)$, exactly as it does at $\beta=1$. 

This paper is organized as follows: Section~\ref{sec:prelim} collects the fractional-calculus facts used throughout the paper; Sections~\ref{sec:line}--\ref{sec:graph} derive the invariance result; Section~\ref{sec:numerics} confirms it numerically; and Section~\ref{sec:discussion} concludes this work and discusses the future 
work on  a genuinely discrete fractional vertex condition.

\section{Preliminaries}\label{sec:prelim}
We collect here the definitions and elementary facts about the Caputo fractional derivative and its Laplace transform that are used, repeatedly and without further comment, in the next Sections~\ref{sec:line}--\ref{sec:graph}.

\begin{definition}[Riemann-Liouville integral and Caputo derivative]\label{def:caputo}
For $\gamma>0$ and $g\in L^1_{\mathrm{loc}}[0,\infty)$, the Riemann-Liouville fractional integral of order $\gamma$ is
\begin{equation}
  I_t^\gamma g(t) = \frac{1}{\Gamma(\gamma)}\int_0^t (t-\tau)^{\gamma-1} g(\tau)\,d\tau, \qquad t>0.
\end{equation}
For $0<\alpha<1$ and $f$ absolutely continuous, the Caputo fractional derivative of order $\alpha$ is
\begin{equation}\label{eq:integral}
    D_t^\alpha f(t) = I_t^{1-\alpha} f'(t) 
    = \frac{1}{\Gamma(1-\alpha)} \int_0^t (t-\tau)^{-\alpha} f'(\tau)\,d\tau.
\end{equation}
We extend this definition to the endpoint $\alpha=1$ by the usual convention $D_t^1 f := f'$, recovering the ordinary derivative; this value is not covered by the integral formula \eqref{eq:integral} (which requires $1-\alpha>0$) but is the natural limiting case and is the one used whenever $\beta=1$ below.
\end{definition}

\begin{lemma}[Laplace transform of the Caputo derivative, {\cite[Section~7.2]{Diethelm2010}}]\label{lem:LTcaputo}
Let $0<\alpha\le1$ and let $f$ be of exponential order with $f'\in L^1_{\mathrm{loc}}[0,\infty)$. 
Then, for $\Real(s)>0$ and with $s^\alpha$ denoting the principal branch,
\begin{equation}\label{eq:LTcaputo0}
    \mathcal{L}\{D_t^\alpha f\}(s) = s^\alpha \widehat{f}(s) - s^{\alpha-1}f(0). 
\end{equation}
In particular, if $f(0)=0$ then $\mathcal{L}\{D_t^\alpha f\}(s)=s^\alpha \widehat{f}(s)$.
\end{lemma}

For $0<\alpha<1$, this is a standard identity of fractional calculus. Here,
we do not reproduce its (short) proof, which combines the convolution theorem, $\mathcal{L}\{t^{\gamma-1}/\Gamma(\gamma)\}(s) = s^{-\gamma}$ for $\Real(s)>0$, $\gamma>0$, with the 
derivative rule $\mathcal{L}\{f'\}(s) = s\widehat{f}(s)-f(0)$. 
At the endpoint $\alpha=1$, the identity is immediate from the same ordinary-derivative rule together with the convention $D_t^1f:=f'$ of Definition~\ref{def:caputo}, and is the case used whenever $\beta=1$ below.

\begin{lemma}[Positivity of $\Real(s^\gamma)$ in the right half-plane]\label{lem:branch}
Let $0<\gamma\le1$. If $\Real(s)>0$, then $\Real(s^\gamma)>0$, for the principal branch of $s^\gamma$.
\end{lemma}
\begin{proof}
Since $\Real(s)>0$, the point $s$ lies strictly in the right half of the complex plane, so we may write it in polar form as 
$s=r\,\mathrm{e}^{\mathrm{i}\theta}$ with modulus $r>0$ and argument 
$\theta\in(-\pi/2,\pi/2)$. 
Recall that the argument is restricted to this open interval precisely because $s$ has positive real part. 
We use the principal branch of the power function, defined on this same right half-plane by
$s^\gamma := r^\gamma \,\mathrm{e}^{\mathrm{i}\gamma\theta}$, 
i.e.\ it raises the modulus to the power $\gamma$ and multiplies the argument by $\gamma$.

Since $0<\gamma\le1$ and $\theta\in(-\pi/2,\pi/2)$, multiplying through by $\gamma$ gives $\gamma\theta\in(-\gamma\pi/2,\gamma\pi/2)\subseteq(-\pi/2,\pi/2)$: 
the rescaled angle $\gamma\theta$ still lies strictly between $-\pi/2$ and $\pi/2$,
and in particular it can never reach $\pm\pi/2$. 
This is exactly the range on which the cosine function is strictly positive, so $\cos(\gamma\theta)>0$.

Finally, taking the real part of $s^\gamma=r^\gamma\, \mathrm{e}^{\mathrm{i}\gamma\theta}=r^\gamma\big(\cos(\gamma\theta)+ \mathrm{i}\sin(\gamma\theta)\big)$ gives $\Real(s^\gamma) = r^\gamma\,\cos(\gamma\theta)$. 
Since $r^\gamma>0$ (as $r>0$) and $\cos(\gamma\theta)>0$ by the previous step, their product is positive, i.e.\ $\Real(s^\gamma)>0$, as claimed. 
We note that the assumption $\gamma\le1$ is essential here: for $\gamma>1$ the angle $\gamma\theta$ could exceed $\pi/2$ in absolute value even when $\theta$ is close to $\pm\pi/2$, making $\cos(\gamma\theta)$ negative and the conclusion false.
\end{proof}

The following application of Lemma~\ref{lem:branch} uses the values of
$\gamma$ and $\beta$  as follows: $\gamma=\beta$ and $\gamma=\beta/2$, both in $(0,1]$  for $\beta\in(0,1]$. 
This guarantees that the exponentials appearing in the Laplace-transformed ordinary differential equations (ODEs) of Sections ~\ref{sec:line}–\ref{sec:graph} decay or grow unambiguously as $x\to\pm\infty$, 
allowing the bounded-solution selection argument used at $\beta=1$ 
in~\cite{Matrasulov2026} to be applied verbatim to  $\beta\in(0,1]$.

We also record, without proof, two standard facts used to justify the manipulations of Sections~\ref{sec:line}--\ref{sec:graph}. 
Differentiation under the Laplace transform in a parameter $x$ is legitimate whenever $u(x,\cdot)$ and $\partial_x u(x,\cdot)$ are, 
locally uniformly in $x$, of exponential order in $t$, and $\partial_x u$ is continuous jointly in $(x,t)$, cf.\ \cite{Doetsch1974}. 
Lerch's uniqueness theorem states that if two functions $f$, $g$, continuous on $[0,\infty)$ and of exponential order, 
have the same Laplace transform on some right half-plane, then $f\equiv g$ \cite{Doetsch1974}.
We use this property to pass from ``the Laplace transform of a difference vanishes'' to ``the difference vanishes'' in the uniqueness arguments of Propositions~\ref{prop:uniqueness} and~\ref{prop:wellposed}.

\begin{assumption}[Standing regularity hypothesis]\label{ass:regularity}
Throughout Sections~\ref{sec:line}--\ref{sec:graph}, we assume that the solutions under consideration exist and admit a Laplace transform in $t$, for every fixed $x$: 
for solutions posed on a domain $x>0$, this transform is assumed bounded as $x\to\infty$, for $\Real(s)>0$; for solutions posed on a domain $x<0$, it is assumed bounded as $x\to-\infty$, for $\Real(s)>0$.
\end{assumption}

This Assumption~\ref{ass:regularity} is not idle. 
For the Cauchy problem on the whole real line, existence and an explicit representation are furnished by Mainardi's fundamental solution~\cite{Mainardi1996} (built from the M-Wright function), whose Gaussian-type spatial decay for every fixed $t>0$ gives exactly the boundedness required by Assumption~\ref{ass:regularity}.
For the associated initial-boundary value problem on a bounded or half-infinite domain, well-posedness in the appropriate weak or strong sense is established by Sakamoto and Yamamoto~\cite{SakamotoYamamoto2011} via eigenfunction expansion and Mittag-Leffler functions.
We do not prove existence again here, relying instead on these references. However, uniqueness follows from Assumption~\ref{ass:regularity} itself, 
as Propositions~\ref{prop:uniqueness} and~\ref{prop:wellposed} below show.

\section{TBC for the Time-Fractional Heat Equation on a Line}\label{sec:line}
We consider the following \textit{time-fractional heat equation}
\begin{equation}\label{eq:fracheat}
    D_t^\beta u(x,t) = \kappa^2\,\partial_x^2 u(x,t), \qquad -\infty<x<\infty,\; t>0.
\end{equation}
Our goal is to relate the boundary trace $u(0,t)$ to the flux $\partial_x u(0,t)$ at an artificial boundary $x=0$, so that a computational domain truncated at $x=0$ reproduces exactly the solution that would be obtained on the full, unbounded line.

\begin{theorem}[Exterior fractional TBC]\label{thm:exterior}
Let $u$ solve \eqref{eq:fracheat} for $x>0$, $t>0$, with quiescent data $u(x,0)=0$ for $x>0$ and $u(0,0)=0$, subject to Assumption~\ref{ass:regularity}. 
Then
\begin{equation} \label{eq:TBCext}
   \frac{\partial u}{\partial x}\bigg|_{x=0} 
   = -\frac{1}{\kappa\,\Gamma(1-\beta/2)}\int_0^t \frac{1}{(t-\tau)^{\beta/2}}\,\frac{\partial u(0,\tau)}{\partial \tau}\,d\tau.
\end{equation}
\end{theorem}
\begin{proof}
Applying the Laplace transform in $t$ to~\eqref{eq:fracheat} and using Lemma~\ref{lem:LTcaputo} with $\alpha=\beta$ and $u(x,0)=0$,
\begin{equation}
      s^\beta\,\widehat{u}(x,s) = \kappa^2\, \widehat{u}''(x,s), \qquad x>0,
\end{equation}
whose general solution is 
\begin{equation*}
    \widehat{u}(x,s) = A(s)\,\mathrm{e}^{xs^{\beta/2}/\kappa} 
    + B(s)\,\mathrm{e}^{-x s^{\beta/2}/\kappa}.
\end{equation*}
The characteristic roots of this constant-coefficient ODE are $\pm\sqrt{s^\beta}/\kappa$; since $\Real(s)>0$ implies, by Lemma~\ref{lem:branch}, that $s^\beta$ itself lies in the right half-plane, its principal square root coincides with $s^{\beta/2}$, justifying the ansatz. 
By Lemma~\ref{lem:branch} (with $\gamma=\beta/2$), $\Real(s^{\beta/2})>0$ for $\Real(s)>0$, so $\mathrm{e}^{xs^{\beta/2}/\kappa}\to\infty$ as $x\to\infty$; Assumption~\ref{ass:regularity} forces $A(s)=0$, hence
\begin{equation*}
   \widehat{u}(x,s) = \widehat{u}(0,s)\, \mathrm{e}^{-x s^{\beta/2}/\kappa}.
\end{equation*}
Differentiating with respect to $x$ (justified by the interchange fact recorded in Section~\ref{sec:prelim}) and evaluating at $x=0$,
\begin{equation}\label{eq:LTflux}
     \widehat{u}'(0,s) = -\frac{s^{\beta/2}}{\kappa}\,\widehat{u}(0,s). 
\end{equation}
Applying the inverse Laplace transform to~\eqref{eq:LTflux}, 
and using Lemma~\ref{lem:LTcaputo} with $\alpha=\beta/2$ (noting $u(0,0)=0$), the multiplier $s^{\beta/2}$ maps to a Caputo fractional derivative of order $\beta/2$ in the time domain, 
yielding~\eqref{eq:TBCext}.
\end{proof}

\begin{remark}[Integrability of the kernel]\label{rem:integrable}
Since $0<\beta/2\le1/2<1$, the kernel $(t-\tau)^{-\beta/2}$ is integrable on $[0,t]$ for every finite $t$: 
\begin{equation*}
    \int_0^t (t-\tau)^{-\beta/2}\,d\tau 
    = t^{1-\beta/2}/(1-\beta/2)<\infty.
\end{equation*}
The operator in~\eqref{eq:TBCext} is thus a genuine (weakly singular but well-defined) convolution acting on any $u(0,\cdot)\in C^1[0,T]$, for any fixed $T>0$; no regularity beyond what is already needed to form $\partial_\tau u(0,\tau)$ pointwise is required for the integral to make sense.
\end{remark}

\begin{proposition}[Uniqueness for the exterior problem]\label{prop:uniqueness}
Suppose $u_1,u_2$ both solve~\eqref{eq:fracheat} for $x>0$ with $u_1(x,0)=u_2(x,0)=0$ and the same boundary trace $u_1(0,t)=u_2(0,t)=g(t)$, both subject to Assumption~\ref{ass:regularity}. Then $u_1\equiv u_2$.
\end{proposition}
\begin{proof}
Let $u_1,u_2$ solve the problem of Theorem~\ref{thm:exterior} with the same boundary trace $g(t):=u_1(0,t)=u_2(0,t)$, and set $w:=u_1-u_2$.
By linearity, $w$ solves~\eqref{eq:fracheat} for $x>0$ with quiescent initial data $w(x,0)=0$, and its Laplace transform satisfies the same ODE as in the proof of Theorem~\ref{thm:exterior},
\begin{equation}\label{eq:unique-ode}
    s^\beta\,\widehat{w}(x,s) = \kappa^2\,\widehat{w}''(x,s), \qquad x>0, 
\end{equation}
with general solution $\widehat{w}(x,s) = A(s)\mathrm{e}^{xs^{\beta/2}/\kappa} + B(s)\mathrm{e}^{-xs^{\beta/2}/\kappa}$. 
The boundedness argument of that proof forces $A(s)=0$, so
\begin{equation}\label{eq:unique-gen}
    \widehat{w}(x,s) = B(s)\,\mathrm{e}^{-x s^{\beta/2}/\kappa}. 
\end{equation}
The boundary traces cancel,
\begin{equation}\label{eq:unique-bc}
     \widehat{w}(0,s) = \widehat{u}_1(0,s) - \widehat{u}_2(0,s) 
      = \widehat{g}(s) - \widehat{g}(s) = 0, 
\end{equation}
and evaluating~\eqref{eq:unique-gen} at $x=0$ gives $B(s)=\widehat{w}(0,s)=0$ by~\eqref{eq:unique-bc}. 
Hence, we have
\begin{equation}\label{eq:unique-vanish}
    \widehat{w}(x,s) \equiv 0 \qquad \text{for every}\quad x>0,\ \Real(s)>0. 
\end{equation}
By Assumption~\ref{ass:regularity}, $w(x,\cdot)$ is of exponential order for each fixed $x$, and it is continuous in $t$ as a difference of classical solutions with continuous data; Lerch's uniqueness theorem~\cite{Doetsch1974} then upgrades~\eqref{eq:unique-vanish} to
\begin{equation*}
     w(x,t) = 0 \qquad \text{for every}\quad x>0,\ t\ge0,
\end{equation*}
i.e.\ $u_1\equiv u_2$.
\end{proof}
The existence of a solution realizing a prescribed boundary trace $g$ is expected to follow from the inversion formula 
\begin{equation*}
   u(x,t) = \mathcal{L}^{-1}\{\widehat{g}(s)\mathrm{e}^{-xs^{\beta/2}/\kappa}\}(t),
\end{equation*}
whose kernel is expressible in closed form via the M-Wright function~\cite{Mainardi1996} and known to decay algebraically rather than exponentially in $t$, a qualitative signature of sub-diffusion documented in detail in~\cite{GorenfloKilbasMainardiRogosin2014}. 
We do not carry out this construction in detail here (in particular, we do not verify that the resulting $u$ satisfies Assumption~\ref{ass:regularity} for a specific, sharply delineated class of boundary data $g$).
A complete existence proof, presumably requiring $g$ Hölder continuous with $g(0)=0$ along the lines of~\cite{SakamotoYamamoto2011}, is left for future work together with the existence question raised after Proposition~\ref{prop:wellposed}.

\begin{corollary}[Interior fractional TBC]\label{thm:interior}
Under the mirrored hypotheses of Theorem~\ref{thm:exterior} on the half-line $x<0$,
\begin{equation}\label{eq:TBCint}
    \frac{\partial u}{\partial x}\bigg|_{x=0} = +\frac{1}{\kappa\,\Gamma(1-\beta/2)}
    \int_0^t \frac{1}{(t-\tau)^{\beta/2}}\,\frac{\partial u(0,\tau)}{\partial \tau}\,d\tau. 
\end{equation}
\end{corollary}
\begin{proof}
Define $v(x,t):=u(-x,t)$ for $x>0$.
Since~\eqref{eq:fracheat} involves only $\partial_x^2$ and $D_t^\beta$, both unaffected by $x\mapsto-x$, $v$ solves~\eqref{eq:fracheat} for $x>0$ with
\begin{equation*}
    v(x,0) = u(-x,0) = 0, \qquad v(0,0) = u(0,0) = 0,
\end{equation*}
and $v$ bounded as $x\to\infty$ (transferred from the boundedness of $u$ as $x\to-\infty$ in Assumption~\ref{ass:regularity}). Theorem~\ref{thm:exterior} therefore applies to $v$:
\begin{equation}\label{eq:TBCint-vproof}
     \frac{\partial v}{\partial x}(0,t)
     = -\frac{1}{\kappa\,\Gamma(1-\beta/2)}\int_0^t \frac{1}{(t-\tau)^{\beta/2}}\,\frac{\partial v(0,\tau)}{\partial \tau}\,d\tau 
     = -\frac{1}{\kappa}\,D_t^{\beta/2}v(0,t). 
\end{equation}
By the chain rule,
\begin{equation}\label{eq:TBCint-chain}
    \partial_x v(x,t) = -\partial_x u(-x,t) \qquad\Longrightarrow\qquad \partial_x v(0,t) = -\partial_x u(0,t), 
\end{equation}
while $v(0,t)=u(-0,t)=u(0,t)$ identically, so
\begin{equation}\label{eq:TBCint-caputo}
D_t^{\beta/2}v(0,t) = D_t^{\beta/2}u(0,t). 
\end{equation}
Substituting~\eqref{eq:TBCint-chain} and~\eqref{eq:TBCint-caputo} into~\eqref{eq:TBCint-vproof} yields
\begin{equation*}
    -\partial_x u(0,t) = -\frac{1}{\kappa}\,D_t^{\beta/2}u(0,t) \qquad\Longrightarrow\qquad \partial_x u(0,t) = \frac{1}{\kappa}\,D_t^{\beta/2}u(0,t).
\end{equation*}
\end{proof}

The sign flip relative to~\eqref{eq:TBCext} is attributed to the reversal of the outward normal, which the reflection $x\mapsto-x$ used above makes precise. 
This non-local-in-time, mildly singular, convolution-type boundary operator of order $\beta/2$ serves.

\section{Extension to the Star Graph}\label{sec:graph}
In this section we consider a simple metric star graph with three bonds $b_j$, $j=1,2,3$:
the incoming bond $b_1\sim(-\infty,0]$, carrying the physical (generally non-quiescent) temperature profile $u_1(x,0)=f_1^I(x)$, and two outgoing bonds $b_{2,3}\sim[0,+\infty)$, initially quiescent. 
On each bond $b_j$ we consider
\begin{equation}\label{eq:fracheatgraph}
    D_t^\beta u_j(x,t) = \kappa_j^2\, \partial_x^2 u_j(x,t), \qquad x\in b_j, 
\end{equation}
with a common fractional order $\beta$ across all three bonds.
We will return to \textit{bond-dependent orders} in Section~\ref{sec:discussion}. 
At the vertex, the temperature $u$ is continuous and the flux obeys the \textit{Kirchhoff-type balance}
\begin{equation} \label{eq:vertexconds}
u_1(0,t)=u_2(0,t)=u_3(0,t), \qquad \kappa_1^2\,\frac{\partial u_1}{\partial x}(0,t) = \kappa_2^2\,\frac{\partial u_2}{\partial x}(0,t) + \kappa_3^2\,\frac{\partial u_3}{\partial x}(0,t).
\end{equation}
We decouple the network into an interior problem on the incoming bond $b_1$ and exterior problems on the two outgoing bonds $b_{2,3}$.

\begin{theorem}[Fractional vertex TBC]\label{thm:vertex}
Under~\eqref{eq:fracheatgraph}--\eqref{eq:vertexconds}, with $u_1(0,\cdot)$ admitting a Laplace transform for $\Real(s)>0$, $u_1(0,0)=0$, and Assumption~\ref{ass:regularity} applied to $u_2,u_3$ on $b_{2,3}$,
we have the fractional vertex TBC
\begin{equation}\label{eq:vertexTBC}
    \frac{\partial u_1}{\partial x}(0,t) = -\frac{\kappa_2+\kappa_3}{\kappa_1^2\,\Gamma(1-\beta/2)}\int_0^t \frac{1}{(t-\tau)^{\beta/2}}\,\frac{\partial u_1}{\partial \tau}(0,\tau)\,d\tau. 
\end{equation}
\end{theorem}
\begin{proof}
Repeating the argument of Theorem~\ref{thm:exterior} on $b_2$, $b_3$, with $\kappa$ replaced by $\kappa_{2,3}$, gives
\begin{equation}\label{eq:branchflux}
     \frac{\partial \widehat{u}_{2,3}}{\partial x}(0,s) = -\frac{s^{\beta/2}}{\kappa_{2,3}}\,\widehat{u}_{2,3}(0,s) = -\frac{s^{\beta/2}}{\kappa_{2,3}}\,\widehat{u}_1(0,s), 
\end{equation}
the second equality using the continuity condition in~\eqref{eq:vertexconds}. Substituting~\eqref{eq:branchflux} into the 
Laplace-transformed Kirchhoff condition,
\begin{equation}\label{eq:vertexLT}
\begin{split}
\kappa_1^2\,\frac{\partial\widehat{u}_1}{\partial x}(0,s) 
&= \kappa_2^2\left(-\frac{s^{\beta/2}}{\kappa_2}\right)\widehat{u}_1(0,s) + \kappa_3^2\left(-\frac{s^{\beta/2}}{\kappa_3}\right)\widehat{u}_1(0,s) \\
&= -(\kappa_2+\kappa_3)\,s^{\beta/2}\,\widehat{u}_1(0,s). 
\end{split}
\end{equation}
By Lemma~\ref{lem:LTcaputo} with $\alpha=\beta/2$, $\mathcal{L}\{D_t^{\beta/2}u_1(0,\cdot)\}(s) = s^{\beta/2}\widehat{u}_1(0,s) - s^{\beta/2-1}u_1(0,0)$; 
since we have assumed $u_1(0,0)=0$, this reduces to $\mathcal{L}\{D_t^{\beta/2}u_1(0,\cdot)\}(s) = s^{\beta/2}\widehat{u}_1(0,s)$, the same clean form obtained for the exterior problem in Theorem~\ref{thm:exterior}. 
Applying the inverse Laplace transform to \eqref{eq:vertexLT} in the form
\begin{equation*}
    \partial_x\widehat{u}_1(0,s) = -\big[(\kappa_2+\kappa_3)/\kappa_1^2\big]s^{\beta/2}\widehat{u}_1(0,s)
\end{equation*}
    therefore gives the TBC~\eqref{eq:vertexTBC}.
\end{proof}

A few points about Theorem~\ref{thm:vertex} deserve comment before we turn to the sum rule itself. 
First, the hypothesis $u_1(0,0)=0$ is not automatic: bond~1 carries the generally non-quiescent profile $u_1(x,0)=f_1^I(x)$ introduced above, so $u_1(0,0)=f_1^I(0)$, and Theorem~\ref{thm:vertex} as stated requires $f_1^I(0)=0$, i.e.\ the incoming datum must vanish exactly at the vertex, though it need not vanish anywhere else on $b_1$. 
If $f_1^I(0)\ne0$, inverting $s^{\beta/2}\widehat{u}_1(0,s)$ produces an additional impulsive term $\dfrac{f_1^I(0)}{\Gamma(1-\beta/2)}\,t^{-\beta/2}$ on the right-hand side of~\eqref{eq:vertexTBC}, coming from the $-s^{\beta/2-1}u_1(0,0)$ term in Lemma~\ref{lem:LTcaputo}. 
This hypothesis holds, in particular, whenever $f_1^I$ is compactly supported away from the vertex or, as in the Gaussian pulse used in Section~\ref{sec:numerics}, decays fast enough that $f_1^I(0)$ is negligible to machine precision; it does not hold for data concentrated at or near the vertex itself, a case Theorem~\ref{thm:vertex} does not cover. The uniqueness result of Proposition~\ref{prop:wellposed} below does not need this hypothesis, since its proof never inverts the Laplace transform back to the time domain. In the case of non-compact data one might explore the ideas outlined in \cite{ehrhardt2008discrete}.

Second, and most importantly for what follows: bond~1 occupies the half-line $x\le0$, the same orientation as the interior problem of Corollary~\ref{thm:interior}, and it is tempting to test transparency by comparing~\eqref{eq:vertexTBC} directly against the interior TBC~\eqref{eq:TBCint} evaluated at $\kappa=\kappa_1$. 
This is not the correct comparison, and doing so would make the two operators' prefactors, $-(\kappa_2+\kappa_3)/\kappa_1^2$ and $+1/\kappa_1$, agree for no positive $\kappa_1$, $\kappa_2$, $\kappa_3$. 
The reason is that~\eqref{eq:TBCint} answers a different question:
it is the condition felt at $x=0$ by a domain $x<0$ that is itself the one carrying the solution, with the half-line $x>0$ discarded. 
Transparency of the junction asks something else: whether the pair of outgoing bonds $b_2$, $b_3$ each of which draws flux from the vertex via the \textit{exterior} construction of Theorem~\ref{thm:exterior}, exactly as in~\eqref{eq:branchflux} is, in aggregate, indistinguishable from a single hypothetical outgoing bond $b_1'$ of diffusivity $\kappa_1$, sharing the vertex with $b_1$ alone and quiescent like $b_2$, $b_3$.
By Theorem~\ref{thm:exterior} applied to $b_1'$ (with $\kappa=\kappa_1$ in place of $\kappa_{2,3}$), such a bond draws exactly the flux
\begin{equation}\label{eq:TBCext1}
   \frac{\partial u_1}{\partial x}(0,t) = -\frac{1}{\kappa_1\,\Gamma(1-\beta/2)}\int_0^t \frac{1}{(t-\tau)^{\beta/2}}\,\frac{\partial u_1}{\partial \tau}(0,\tau)\,d\tau,
\end{equation}
i.e.\ the \textit{exterior} relation~\eqref{eq:TBCext} with $\kappa=\kappa_1$, carrying the same negative sign as~\eqref{eq:vertexTBC} and as~\eqref{eq:branchflux} itself  not the interior relation~\eqref{eq:TBCint}. 
This is the comparison used in the Proposition~\ref{prop:invariance}.

\begin{proposition}[$\beta$-invariance of the sum rule]\label{prop:invariance}
Under the hypotheses of Theorem~\ref{thm:vertex}, the vertex operator~\eqref{eq:vertexTBC} coincides with the single-bond exterior relation~\eqref{eq:TBCext1} above (Theorem~\ref{thm:exterior} with $\kappa=\kappa_1$), for every $\beta\in(0,1]$, if and only if
\begin{equation}\label{eq:sumrule2}
    \kappa_1 = \kappa_2 + \kappa_3. 
\end{equation}
\end{proposition}
\begin{proof}
Both~\eqref{eq:vertexTBC} and~\eqref{eq:TBCext1} are built from the same non-local kernel $D_t^{\beta/2}u_1(0,\cdot)$ (the latter, like the former, under the standing hypothesis $u_1(0,0)=0$ of Theorem~\ref{thm:vertex}, needed to invert $s^{\beta/2}\widehat{u}_1(0,s)$ cleanly by Lemma~\ref{lem:LTcaputo}); the fractional order $\beta$ enters both sides only through this common factor, which is untouched by the Kirchhoff algebra of Theorem~\ref{thm:vertex}, since that algebra acts solely on the $\kappa_j$-dependent prefactors. 
The two expressions therefore coincide, for every $t$ and every admissible boundary trace, if and only if their scalar prefactors agree: 
$-(\kappa_2+\kappa_3)/\kappa_1^2 = -1/\kappa_1 \iff \kappa_1=\kappa_2+\kappa_3$ (multiplying both sides by $-\kappa_1^2<0$). Since this equivalence does not involve $\beta$, it holds for every $\beta\in(0,1]$.
\end{proof}

This $\beta$-invariance of the sum rule~\eqref{eq:sumrule2} is, to our knowledge, a new observation, and is the central analytical result of this note:
\textit{the transparency of a branched thermal junction is governed entirely by the network's geometry and diffusivities, independent of the memory structure the fractional order $\beta$ of the underlying transport law}.

\begin{corollary}[$n$-bond star graph]\label{thm:nbond}
Consider an $n$-bond star graph with one incoming bond $b_1$ and $n-1$ outgoing bonds $b_2,\ldots,b_n$, each of diffusivity $\kappa_j$, sharing the common order $\beta$, 
and governed by~\eqref{eq:fracheatgraph} with the $n$-bond Kirchhoff condition 
$\kappa_1^2\partial_x u_1(0,t) = \sum_{j=2}^n \kappa_j^2\,\partial_x u_j(0,t)$. 
Then
\begin{equation}\label{eq:nbondTBC}
  \frac{\partial u_1}{\partial x}(0,t) 
  = -\frac{\sum_{j=2}^n\kappa_j}{\kappa_1^2\,\Gamma(1-\beta/2)}\int_0^t\frac{1}{(t-\tau)^{\beta/2}}\frac{\partial u_1}{\partial\tau}(0,\tau)\,d\tau, 
\end{equation}
which, under the standing hypothesis $u_1(0,0)=0$, coincides with the single-bond exterior relation~\eqref{eq:TBCext1} for every $\beta\in(0,1]$ if and only if $\kappa_1=\sum_{j=2}^n\kappa_j$.
\end{corollary}
\begin{proof}
Repeating~\eqref{eq:branchflux} on every outgoing bond,
\begin{equation}\label{eq:nbond-branch}
\partial_x\widehat{u}_j(0,s) = -\frac{s^{\beta/2}}{\kappa_j}\,\widehat{u}_1(0,s), \qquad j=2,\ldots,n. 
\end{equation}
Substituting~\eqref{eq:nbond-branch} into the $n$-bond Kirchhoff condition $\kappa_1^2\partial_x u_1(0,t) = \sum_{j=2}^n \kappa_j^2\,\partial_x u_j(0,t)$, in Laplace-transformed form,
\begin{equation}\label{eq:nbond-LT}
  \kappa_1^2\,\frac{\partial\widehat{u}_1}{\partial x}(0,s) 
  = \sum_{j=2}^n \kappa_j^2\left(-\frac{s^{\beta/2}}{\kappa_j}\right)\widehat{u}_1(0,s) 
  = -\Big(\sum_{j=2}^n\kappa_j\Big)\,s^{\beta/2}\,\widehat{u}_1(0,s), 
\end{equation}
i.e.\ $\partial_x\widehat{u}_1(0,s) = -\big[(\sum_{j=2}^n\kappa_j)/\kappa_1^2\big]s^{\beta/2}\widehat{u}_1(0,s)$. 
By Lemma~\ref{lem:LTcaputo} with $u_1(0,0)=0$, $\mathcal{L}\{D_t^{\beta/2}u_1(0,\cdot)\}(s)=s^{\beta/2}\widehat{u}_1(0,s)$, 
so inverting the Laplace transform in~\eqref{eq:nbond-LT} gives~\eqref{eq:nbondTBC}.

Comparing prefactors with~\eqref{eq:TBCext1},
\begin{equation}\label{eq:nbond-prefactor}
    -\frac{\sum_{j=2}^n\kappa_j}{\kappa_1^2} = -\frac{1}{\kappa_1} \qquad\Longleftrightarrow\qquad \kappa_1 = \sum_{j=2}^n\kappa_j, 
\end{equation}
by the same algebra as in the proof of Proposition~\ref{prop:invariance}, with $\kappa_2+\kappa_3$ replaced by $\sum_{j=2}^n\kappa_j$.
\end{proof}
Corollary~\ref{thm:nbond} shows that the sum rule and its $\beta$-invariance generalize verbatim from three bonds to $n$ bonds; the only change is that the sum on the right runs over all outgoing bonds. 
We will discuss further extensions, and their limits, in Section~\ref{sec:discussion}.

\begin{proposition}[Uniqueness for the coupled vertex problem]\label{prop:wellposed}
Let $f_1^I\in L^1(-\infty,0)$ be bounded (e.g.\ compactly supported, or decaying faster than any polynomial, or merely $L^1$ and bounded, such as a Gaussian).
The coupled problem consisting of~\eqref{eq:fracheatgraph} on $b_1$, $b_2$, $b_3$ with $u_1(x,0)=f_1^I(x)$, $u_2(x,0)=u_3(x,0)=0$, and the vertex conditions~\eqref{eq:vertexconds}, has at most one solution satisfying Assumption~\ref{ass:regularity}.
\end{proposition}
\begin{proof}
Combining bond~1's own equation with its initial datum, the Laplace transform of~\eqref{eq:fracheatgraph} on $b_1$ reads, by Lemma~\ref{lem:LTcaputo} with the inhomogeneous term retained,
\begin{equation}\label{eq:wp-ode}
   \widehat{u}_1''(x,s) - \lambda(s)^2\,\widehat{u}_1(x,s) = h(x,s), \qquad x<0, 
\end{equation}
where we abbreviate
\begin{equation*}
   \lambda(s) := \frac{s^{\beta/2}}{\kappa_1}, \qquad 
   h(x,s) := -\frac{s^{\beta-1}}{\kappa_1^2}\,f_1^I(x),
\end{equation*}
subject to the Robin-type condition
\begin{equation}\label{eq:wp-robin}
     \widehat{u}_1'(0,s) = -\frac{\kappa_2+\kappa_3}{\kappa_1^2}\,s^{\beta/2}\,\widehat{u}_1(0,s) 
\end{equation}
from~\eqref{eq:vertexLT}, together with boundedness as $x\to-\infty$. By Lemma~\ref{lem:branch}, $\Real(\lambda(s))>0$ for $\Real(s)>0$.

Since $f_1^I\in L^1(-\infty,0)$, 
$h(\cdot,s)=-\kappa_1^{-2}s^{\beta-1}f_1^I(\cdot)\in L^1(-\infty,0)$ for every fixed $s$ with $\Real(s)>0$, and a particular solution of~\eqref{eq:wp-ode}, 
bounded as $x\to-\infty$, is furnished by the standard Green's function for $y''-\lambda^2y=h$ on the whole line,
\begin{equation}\label{eq:wp-particular}
   \widehat{u}_{1,\mathrm{part}}(x,s) = -\frac{1}{2\lambda(s)}
   \int_{-\infty}^0 \mathrm{e}^{-\lambda(s)|x-\xi|}\,h(\xi,s)\,d\xi. 
\end{equation}
The general solution of~\eqref{eq:wp-ode} bounded as $x\to-\infty$ is then
\begin{equation}\label{eq:wp-general}
    \widehat{u}_1(x,s) = \widehat{u}_{1,\mathrm{part}}(x,s) + C(s)\,\mathrm{e}^{\lambda(s)x}, 
\end{equation}
where the second homogeneous mode, $\mathrm{e}^{-\lambda(s)x}$, has already been discarded: since $\Real(\lambda(s))>0$, $\mathrm{e}^{-\lambda(s)x}\to\infty$ as $x\to-\infty$, whereas $\mathrm{e}^{\lambda(s)x}\to0$, mirroring the boundedness argument of Corollary~\ref{thm:interior}. 
The constant $C(s)$ remains to be fixed by the Robin condition~\eqref{eq:wp-robin}.

Next, differentiating~\eqref{eq:wp-general} and evaluating both $\widehat{u}_1$ and $\widehat{u}_1'$ at $x=0$,
\begin{equation*}
   \widehat{u}_1(0,s) = \widehat{u}_{1,\mathrm{part}}(0,s) + C(s), \qquad
   \widehat{u}_1'(0,s) = \widehat{u}_{1,\mathrm{part}}'(0,s) + C(s)\,\lambda(s).
\end{equation*}
Substituting into~\eqref{eq:wp-robin} and collecting the terms in $C(s)$,
\begin{equation} \label{eq:wp-coeffeq}
    C(s)\biggl[\lambda(s) + \frac{\kappa_2+\kappa_3}{\kappa_1^2}\,s^{\beta/2}\biggr] 
   = -\frac{\kappa_2+\kappa_3}{\kappa_1^2}\,s^{\beta/2}\,\widehat{u}_{1,\mathrm{part}}(0,s) - \widehat{u}_{1,\mathrm{part}}'(0,s).
\end{equation}
Using $\lambda(s)=s^{\beta/2}/\kappa_1$, the bracket on the left simplifies to
\begin{equation} \label{eq:wp-coeff}
    \lambda(s) + \frac{\kappa_2+\kappa_3}{\kappa_1^2}\,s^{\beta/2} = \frac{s^{\beta/2}}{\kappa_1} + \frac{(\kappa_2+\kappa_3)\,s^{\beta/2}}{\kappa_1^2} = \frac{s^{\beta/2}\,(\kappa_1+\kappa_2+\kappa_3)}{\kappa_1^2},
\end{equation}
which is nonzero for every $s$ with $\Real(s)>0$, since $\kappa_1$, $\kappa_2$, $\kappa_3>0$ and $s^{\beta/2}\neq0$ there.
Equation~\eqref{eq:wp-coeffeq} therefore determines $C(s)$ uniquely, and with it, via~\eqref{eq:wp-general}, $\widehat{u}_1(x,s)$ for every $x<0$ and every $s$ with $\Real(s)>0$.

If $u_1$, $\tilde{u}_1$ both solve the coupled problem, their Laplace transforms consequently agree for every $x<0$ and every such $s$; 
Lerch's uniqueness theorem \cite{Doetsch1974} then gives $u_1\equiv\tilde{u}_1$, and~\eqref{eq:branchflux} propagates this identity to $u_2\equiv\tilde{u}_2$, $u_3\equiv\tilde{u}_3$.
\end{proof}
\begin{remark}
    No hypothesis on $f_1^I(0)$ is needed in Proposition~\ref{prop:wellposed}, in contrast to Theorem~\ref{thm:vertex}: the argument below never inverts the Laplace transform back to the time domain.
\end{remark}

We do not carry out the corresponding existence argument (the explicit Green's-function construction is routine but lengthy) and leave a full existence-and-uniqueness theory for the coupled vertex problem, analogous to~\cite{SakamotoYamamoto2011} on a single interval but adapted to graph vertex conditions, for future work.

\section{Numerical Results}\label{sec:numerics}
To validate the analytical reflectionless condition, we discretize~\eqref{eq:fracheatgraph} on each bond with the standard implicit L1 scheme for the Caputo derivative~\cite{LinXu2007}, combined with a backward (fully implicit) discretization of the spatial Laplacian. With $u_i^n\approx u(x_i,t_n)$, $t_n=n\Delta t$, and weights
\begin{equation*}
   b_j = (j+1)^{1-\beta} - j^{1-\beta}, \qquad j=0,1,2,\ldots,
\end{equation*}
the L1 approximation
\begin{equation*}
     D_t^\beta u(x_i,t_{n+1}) \;\approx\; \frac{1}{\Gamma(2-\beta)\,\Delta t^{\beta}}\sum_{j=0}^{n} b_j\big(u_i^{n+1-j}-u_i^{n-j}\big)
\end{equation*}
gives, at every interior spatial node, a linear equation for $u_i^{n+1}$ in terms of the full nodal history $u_i^0,\ldots,u_i^n$ and the implicit spatial neighbors $u_{i\pm1}^{n+1}$; the scheme is unconditionally stable and converges at rate $2-\beta$ in time~\cite{LinXu2007}. 
Because the weights $b_j$ multiply the entire history at every time step, the scheme's cost and memory footprint scale as $O(N_t^2)$ in the number of time steps $N_t$, rather than $O(N_t)$ as for a classical (integer-order) implicit scheme; for the domain sizes and time horizons used below this remains modest, but a fast convolution technique (e.g.\ based on the sum-of-exponentials approximation of the kernel) would be needed for substantially longer simulations.
The vertex conditions~\eqref{eq:vertexconds} involve only values at the same time level $t_{n+1}$ and are discretized as in the classical Crank-Nicolson scheme: a one-sided finite difference for the Kirchhoff flux, and a single shared vertex unknown enforcing continuity. 
The fractional order $\beta$ affects only the bulk (interior-node) equations, not the vertex-coupling algebra.

Before turning to the sum rule itself, we checked the implementation against the two facts that constrain it independently of any graph structure: 
the $L^1$ weights $b_j$ must reproduce the exact Caputo derivative of a simple test function (we used $f(t)=t^3$, for which $D_t^\beta f$ is known in closed form via $D_t^\beta t^p = \Gamma(p+1)/\Gamma(p+1-\beta)\,t^{p-\beta}$), and the resulting scheme, applied to~\eqref{eq:fracheat} on a single unbranched line, must converge at the expected rates as $\Delta t\to0$ (holding $\Delta x$ fixed and small) and as $\Delta x\to0$ (holding $\Delta t$ fixed and small). 
This check caught a genuine indexing error in an earlier version of our implementation the history differences entering the $L^1$ sum were being paired with the weights $b_j$ in the wrong temporal order which the sum-rule comparison below alone would not have revealed, 
since it degrades the reference and the star-graph solution consistently and largely cancels in their difference. 

Figure~\ref{fig:convergence} shows, for $\beta=0.7$, the corrected scheme's self-convergence: refining $\Delta t$ with $\Delta x=0.025$ fixed gives successive-difference ratios corresponding to an observed rate of about $1.03$--$1.05$, somewhat below the asymptotic rate $2-\beta=1.3$ of~\cite{LinXu2007}; this shortfall is expected: 
solutions of time-fractional equations typically retain only limited time-regularity near $t=0$ even for smooth initial data, a fact established rigorously by Stynes, O'Riordan and Gracia~\cite{StynesORiordanGracia2017} for the Caputo reaction-diffusion problem, who derive sharp pointwise bounds on the solution's time-derivatives showing a weak singularity at $t=0$ and prove that this singularity, not the formal $O(\Delta t^{2-\beta})$ truncation order of the L1 scheme, governs the observed convergence rate on a uniform mesh; the asymptotic rate is recovered only once $\Delta t$ is small relative to this initial-layer scale, or on an appropriately graded mesh. 
Refining $\Delta x$ with $\Delta t=0.02$ fixed gives an observed rate of about $2.0$--$2.2$, matching the expected second-order spatial accuracy cleanly. 

\begin{figure}[h]
\centering
\includegraphics[width=\linewidth]{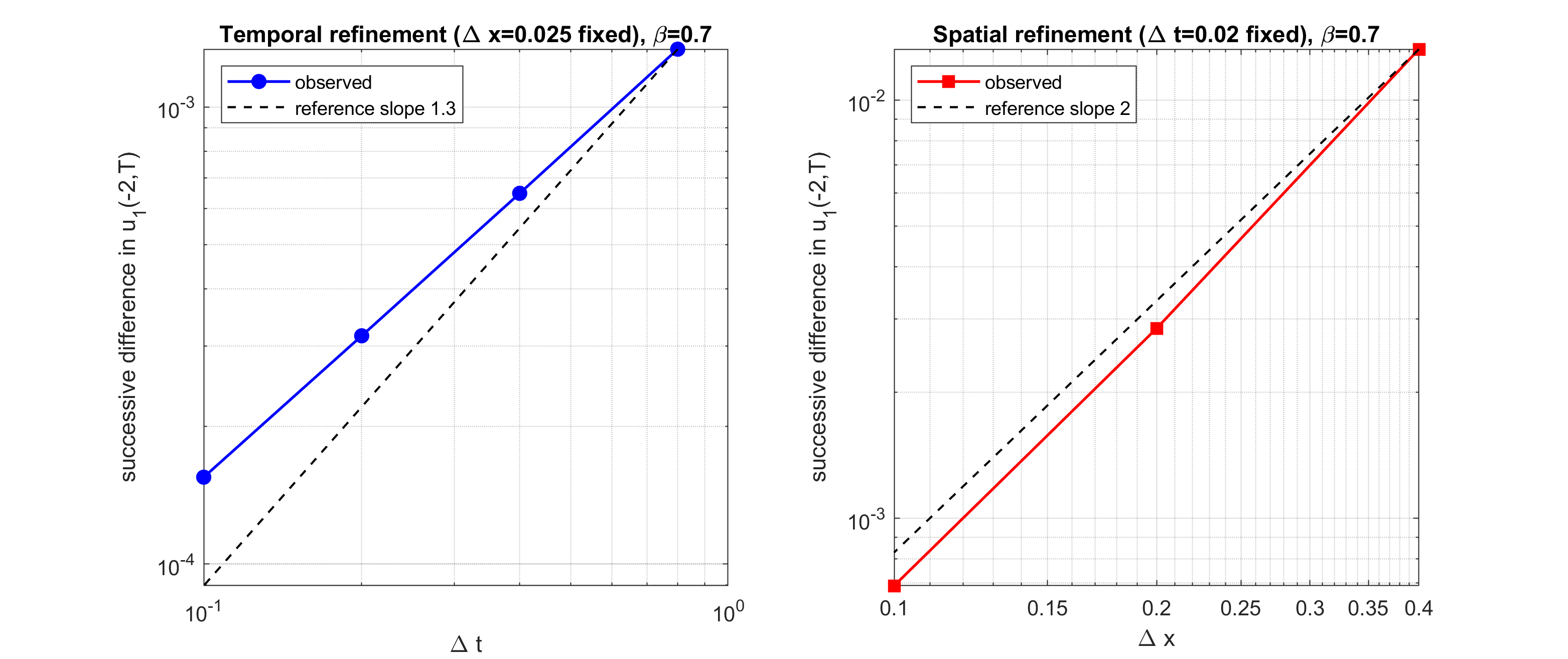}
\caption{Self-convergence of the L1/implicit scheme for $\beta=0.7$, $\kappa_1=0.169$, measured as the successive difference in $u_1(-2,T)$ under refinement. Left: temporal refinement ($\Delta x=0.025$ fixed), dashed line of reference slope $2-\beta$. Right: spatial refinement ($\Delta t=0.02$ fixed), dashed line of reference slope $2$.}
\label{fig:convergence}
\end{figure}

The convergence study above was carried out on the uncoupled reference line, to isolate the bulk $L^1$ scheme's own accuracy from any effect of the vertex discretization. 
We repeated it directly on the coupled star-graph system~\eqref{eq:fracheatgraph}--\eqref{eq:vertexconds}, 
with the sum rule satisfied ($\kappa_1=\kappa_2+\kappa_3=0.169$), probing $u_1$ at two locations: an interior point on bond~1 ($x=-2$, away from the vertex) and the vertex value itself ($x=0$).
Table~\ref{tab:starconv} reports the resulting rates. 
Away from the vertex, the coupled system converges at essentially the same rates as the uncoupled reference line (temporal $\approx1.03$--$1.05$, spatial $\approx2.04$--$2.22$), confirming that the vertex coupling does not degrade the bulk scheme's accuracy. At the vertex itself, however, the spatial rate drops to $\approx1.21$--$1.35$: 
this confirms, rather than merely asserts, that the one-sided finite difference used for the Kirchhoff flux~\eqref{eq:vertexconds} is locally the dominant source of spatial error near the junction, while the bulk of each bond retains full second-order accuracy.

\begin{table}[h]
\centering
\caption{Observed convergence rates on the coupled star-graph system, $\beta=0.7$, $\kappa_1=\kappa_2+\kappa_3=0.169$, probing $u_1$ away from the vertex ($x=-2$) and at the vertex ($x=0$).}
\label{tab:starconv}
\begin{tabular}{@{}lcc@{}}
\toprule
Refinement & rate, $x=-2$ & rate, $x=0$ (vertex) \\
\midrule
Temporal ($\Delta x=0.025$ fixed), expect $2-\beta=1.3$ & $1.03$--$1.05$ & --- \\
Spatial ($\Delta t=0.02$ fixed), bulk expects $2$ & $2.04$--$2.22$ & $1.21$--$1.35$ \\
\bottomrule
\end{tabular}
\end{table}

To confirm that the reduced uniform-mesh rate is indeed attributable to the $t=0$ initial layer rather than to some other deficiency of our implementation, we repeated the temporal refinement study on a graded mesh $t_k=(k/N)^r\,T$, $k=0,\ldots,N$, using the general (non-uniform) form of the L1 approximation, obtained by replacing the constant weights $\Delta t^{-\beta}b_j/\Gamma(2-\beta)$ with the exact integrals of $(t_{n+1}-\tau)^{-\beta}$ over each non-uniform subinterval. 
Stynes, O'Riordan and Gracia~\cite{StynesORiordanGracia2017} show that a grading exponent $r\ge(2-\beta)/\beta$ suffices to recover the full rate $2-\beta$; we used $r=(2-\beta)/\beta\approx1.86$ for $\beta=0.7$. 
Figure~\ref{fig:graded} compares the resulting successive-difference rates against the uniform-mesh case of Figure~\ref{fig:convergence}: the graded mesh gives an observed rate of about $1.28$--$1.29$, closely matching the theoretical rate $2-\beta=1.3$, against $1.03$--$1.06$ on the uniform mesh with the same number of time steps. 
This confirms both that our diagnosis of the reduced uniform-mesh rate is correct, and that the corrected implementation behaves exactly as the cited theory predicts once the initial-layer singularity is properly resolved.

\begin{figure}[h]
\centering
\includegraphics[width=0.75\linewidth]{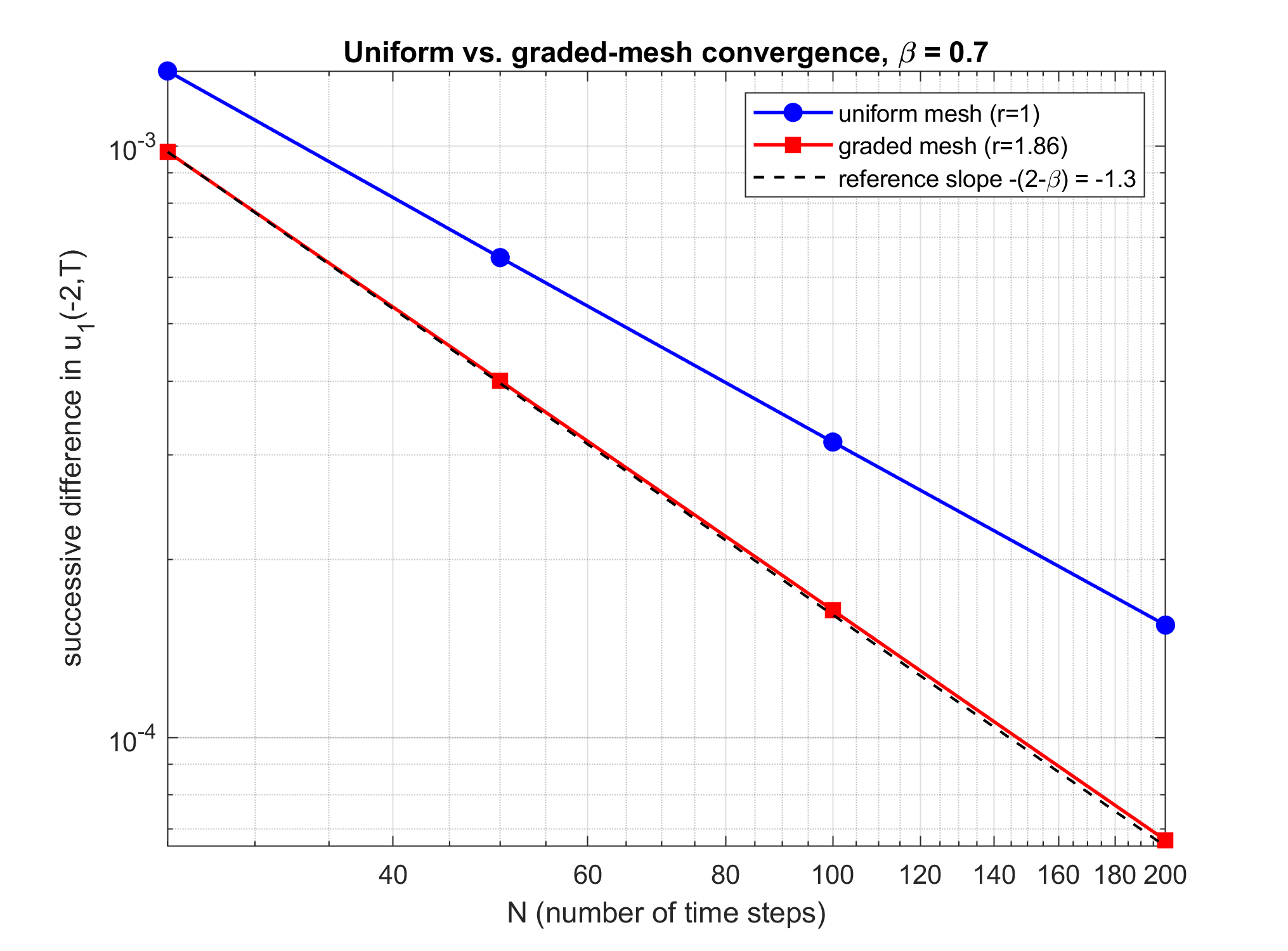}
\caption{Successive-difference convergence in $N$ (number of time steps) for $\beta=0.7$, $\kappa_1=0.169$: uniform mesh ($r=1$) vs.\ graded mesh ($r=(2-\beta)/\beta$), against a reference slope $2-\beta$.}
\label{fig:graded}
\end{figure}

Rather than discretizing the non-local operator~\eqref{eq:vertexTBC} itself  a task we return to in Section~\ref{sec:discussion}  we validate the sum rule directly at the level of the coupled partial differential equations (PDEs) system, mirroring the strategy of Figures~3--4 of~\cite{Matrasulov2026}: 
we compare the numerical solution $u_1(x,t)$ on bond~1 of the star graph against a reference solution $u_1^{\mathrm{ref}}(x,t)$, obtained by solving the same time-fractional heat equation~\eqref{eq:fracheat}, with diffusivity $\kappa_1$, on a single unbranched line (no vertex, no Kirchhoff coupling), truncated symmetrically and far enough that the truncation itself is not visible on the time horizon considered. 
Both solutions use the same L1/implicit scheme, the same $\Delta x,\Delta t$, and the same Gaussian initial datum on bond~1 (zero on the positive half of the reference line, matching the quiescent bonds $b_2,b_3$). 
If the sum rule makes the junction transparent, the two solutions should coincide on $x\le0$; any deviation quantifies backflow. 
We emphasize that this validates the sum rule's effect at the level of the coupled PDE system~\eqref{eq:fracheatgraph}--\eqref{eq:vertexconds} and its standard Kirchhoff condition, already established analytically in Proposition~\ref{prop:invariance}: it is not a direct numerical discretization of the non-local operator~\eqref{eq:vertexTBC} itself (which, as noted above, we do not discretize here), so it should not be read as an independent numerical confirmation of Theorem~\ref{thm:vertex}'s specific operator identity beyond what the analytical argument already establishes.
Here, we used $\kappa_2=0.089$, $\kappa_3=0.08$, a Gaussian pulse $u_1(x,0)=\mathrm{e}^{-(x+2)^2/0.2}$ on bond~1, $\Delta x=0.2$--$0.25$, $\Delta t=0.15$, and truncated domains at $|x|\le12$--$15$ (verified far enough not to affect the reported time horizon).

Figures~\ref{fig:evol_sat} and~\ref{fig:evol_broken} show the full evolution of the temperature profile on all three bonds, for $\beta=0.7$, in the reflectionless ($\kappa_1=\kappa_2+\kappa_3=0.169$) and non-reflectionless ($\kappa_1=0.35$) regimes, respectively (note the different vertical scales on bond~1 vs.\ bonds~2,3, since the transmitted pulse spreads over a much longer channel and its peak amplitude on $b_2,b_3$ is correspondingly smaller). 
In the reflectionless case, the incoming pulse splits smoothly between $b_2$ and $b_3$ with no visible distortion on bond~1 beyond ordinary sub-diffusive spreading.
When the sum rule is broken, bond~1 instead develops a visible secondary bulge near the vertex ($t=15,30,60$ in Figure~\ref{fig:evol_broken}) that is absent in Figure~\ref{fig:evol_sat}: 
this is the reflected component of the pulse, directly visible in the spatial profile rather than only in the aggregate diagnostic $\mathcal{E}(\kappa_1)$ introduced below.

\begin{figure}[h]
\centering
\includegraphics[width=\linewidth]{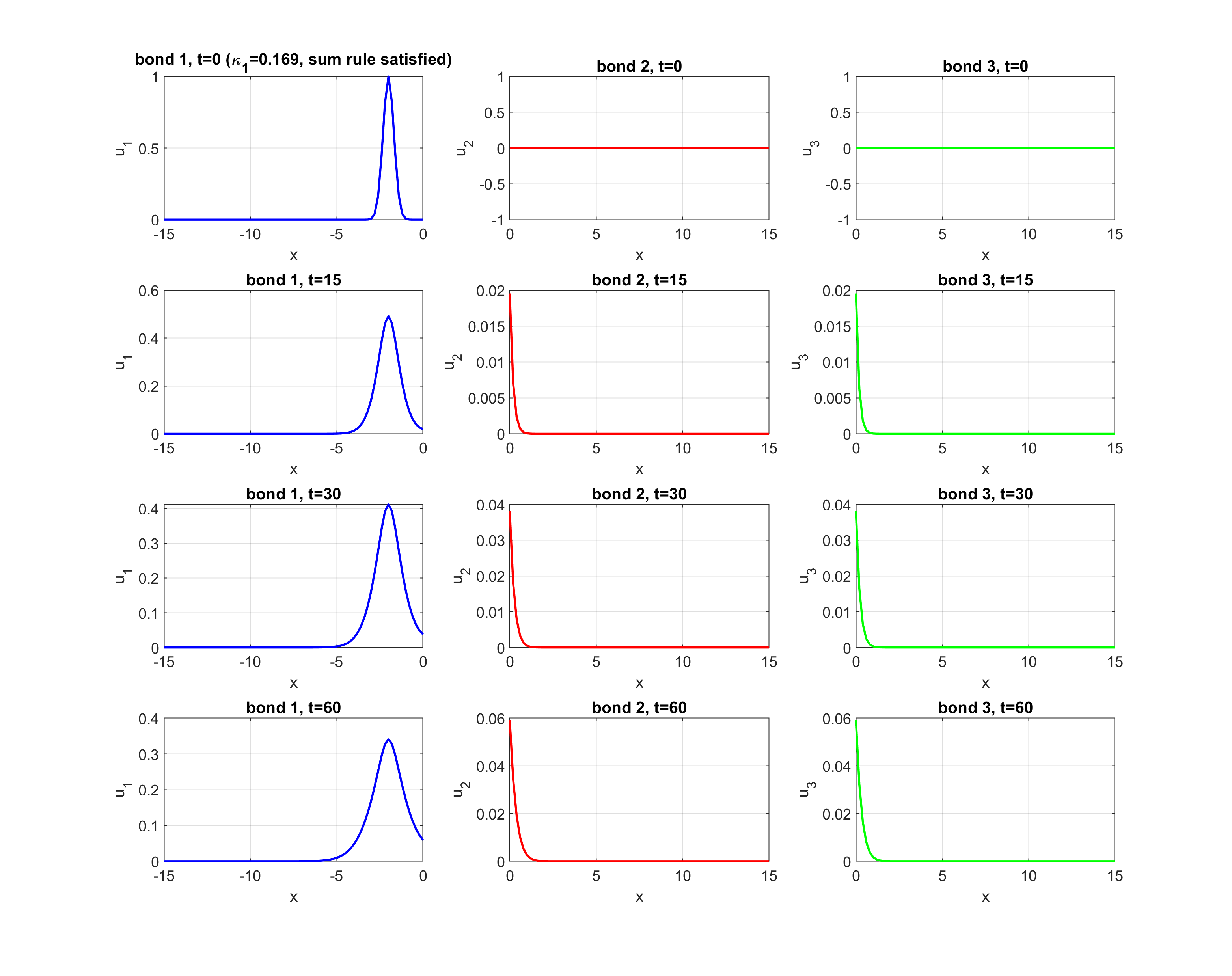}
\caption{Evolution of the temperature profile on all three bonds, $\beta=0.7$, $\kappa_1=\kappa_2+\kappa_3=0.169$ (sum rule satisfied). Columns: bonds 1, 2, 3; rows: $t=0,15,30,60$.}
\label{fig:evol_sat}
\end{figure}

\begin{figure}[h]
\centering
\includegraphics[width=\linewidth]{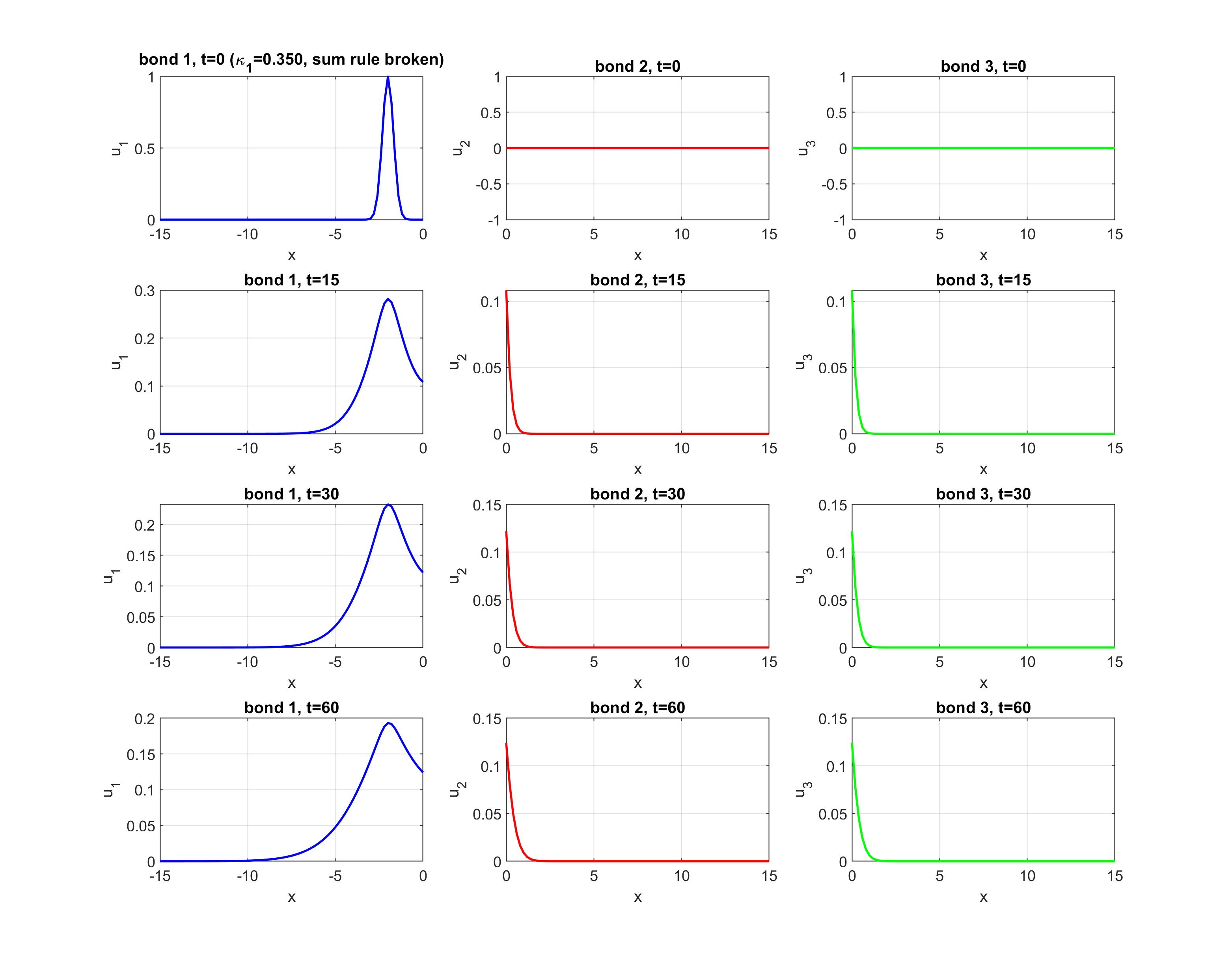}
\caption{Evolution of the temperature profile on all three bonds, $\beta=0.7$, $\kappa_1=0.35\neq\kappa_2+\kappa_3$ (sum rule broken). Same layout as Figure~\ref{fig:evol_sat}; note the secondary bulge developing on bond~1 near the vertex at $t=15,30,60$, absent under the sum rule.}
\label{fig:evol_broken}
\end{figure}

We measured the relative \textit{$L^2$ transparency error}
\begin{equation*}
    \mathcal{E}(\kappa_1) = \frac{\big\|u_1(\cdot,T)-u_1^{\mathrm{ref}}(\cdot,T)\big\|_{L^2(x\le0)}}{\big\|u_1^{\mathrm{ref}}(\cdot,T)\big\|_{L^2(x\le0)}}
\end{equation*}
at a fixed final time $T$, as a function of $\kappa_1$, for several values of $\beta$.
Figure~\ref{fig:snapshots} shows the temperature profile on bond~1 at several times for $\beta=0.7$ and $\kappa_1=\kappa_2+\kappa_3=0.169$: the star-graph solution and the unbranched reference are visually indistinguishable throughout the evolution.
\begin{figure}[h]
\centering
\includegraphics[width=\linewidth]{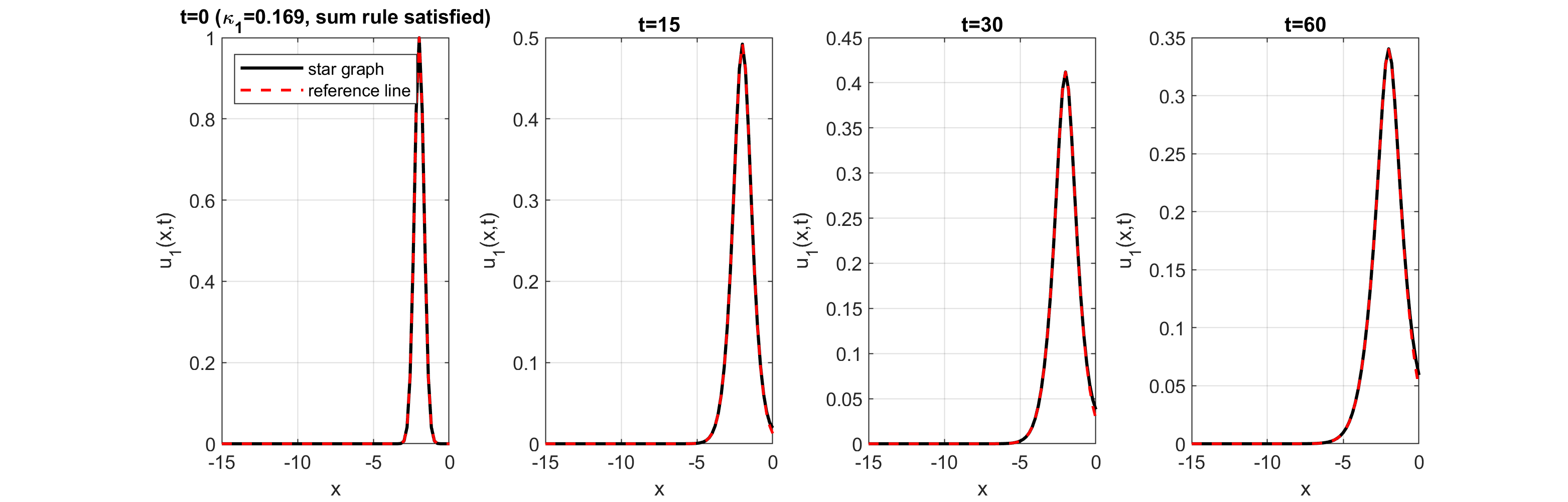}
\caption{Temperature profile on bond~1, star-graph simulation (black) vs.\ unbranched reference (red dashed), for $\beta=0.7$ and $\kappa_1=\kappa_2+\kappa_3=0.169$ (sum rule satisfied). $\kappa_2=0.089$, $\kappa_3=0.08$.}
\label{fig:snapshots}
\end{figure}
For contrast, Figure~\ref{fig:snapshots_broken} repeats this comparison with the sum rule broken ($\kappa_1=0.35$): a small but visible deviation develops between the star-graph solution and the reference line as the pulse reaches and passes through the vertex, consistent with the elevated $\mathcal{E}(\kappa_1)$ reported in Table~\ref{tab:results} for this parameter choice.

\begin{figure}[h]
\centering
\includegraphics[width=\linewidth]{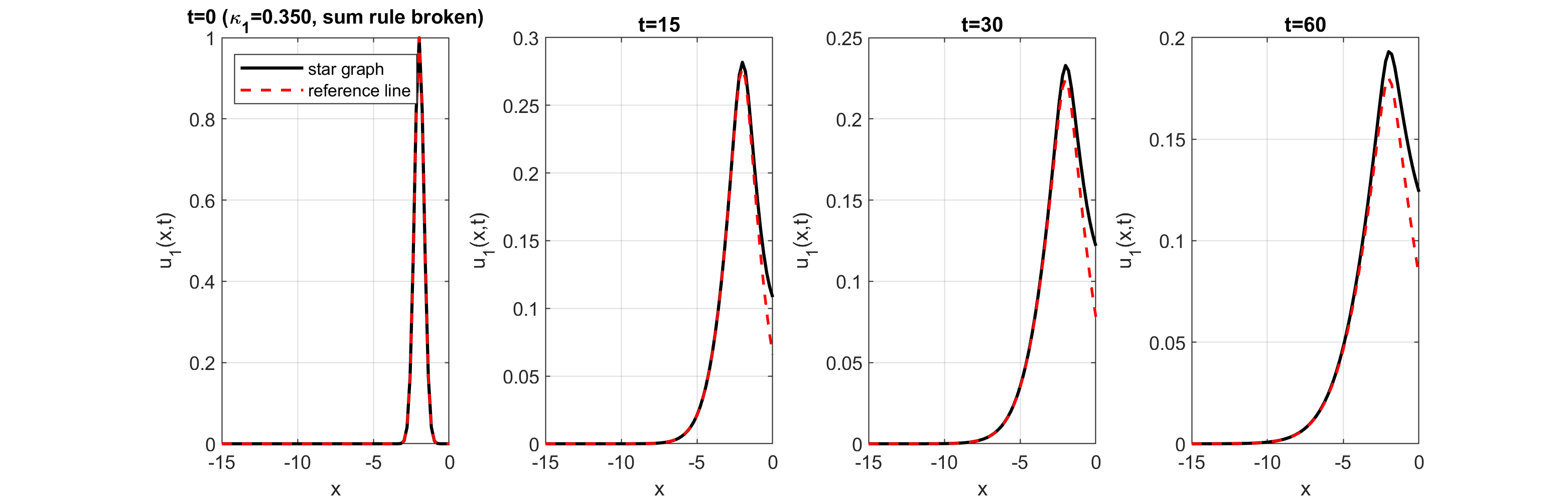}
\caption{Same comparison as Figure~\ref{fig:snapshots}, with the sum rule broken ($\kappa_1=0.35$, $\beta=0.7$). A small deviation between the star-graph solution and the reference line is visible near the peak at $t=15,30,60$.}
\label{fig:snapshots_broken}
\end{figure}

Figure~\ref{fig:error} shows the relative transparency error $\mathcal{E}(\kappa_1)$ as a function of $\kappa_1$, for $\beta=0.99,0.7,0.5$, with $\kappa_2,\kappa_3$ fixed as above. 
For every $\beta$ tested, $\mathcal{E}$ attains its minimum at, or immediately next to, 
$\kappa_1=\kappa_2+\kappa_3=0.169$ (vertical dotted line), confirming numerically that the sum rule remains the transparency-optimal choice of $\kappa_1$ independently of $\beta$. 
Table~\ref{tab:results} reports the corresponding numbers for $\kappa_1=0.169$ (sum rule) versus $\kappa_1=0.35$, across five values of $\beta$; the relative error is consistently three-to-five times smaller when the sum rule holds. 
The residual error at $\kappa_1=0.169$ is not exactly zero, which is expected: it is a discretization artifact of the L1/implicit scheme, and Table~\ref{tab:starconv} above confirms that the one-sided differencing at the vertex is indeed its dominant local source.

\begin{figure}[h]
\centering
\includegraphics[width=0.7\linewidth]{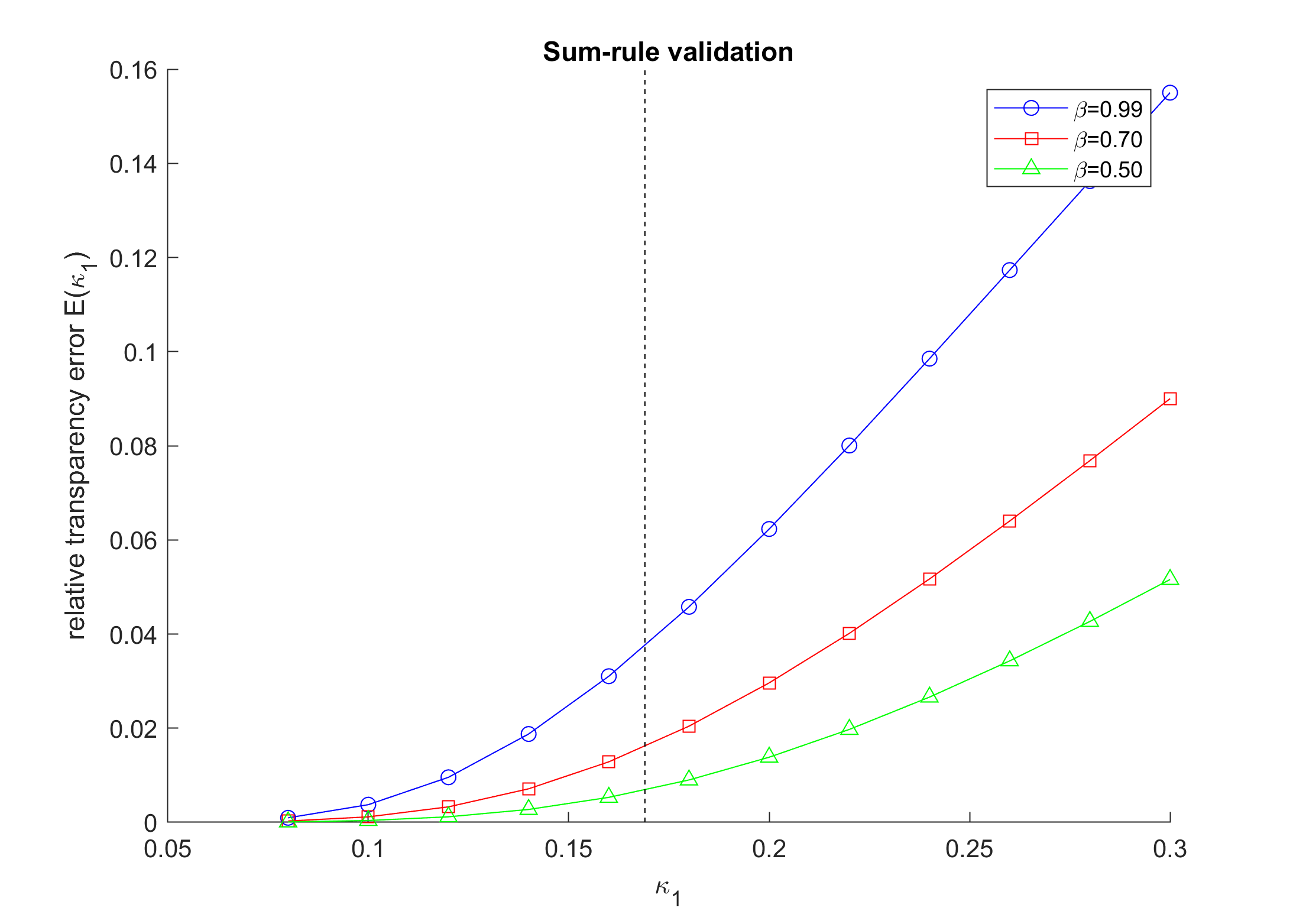}
\caption{Relative transparency error $\mathcal{E}(\kappa_1)$ vs.\ $\kappa_1$, for three values of $\beta$. The dotted vertical line marks $\kappa_1=\kappa_2+\kappa_3=0.169$.}
\label{fig:error}
\end{figure}

\begin{table}[h]
\centering
\caption{Relative transparency error $\mathcal{E}$ at the final simulated time, sum rule satisfied ($\kappa_1=0.169$) vs.\ broken ($\kappa_1=0.35$), for several $\beta$.}
\label{tab:results}
\begin{tabular}{@{}ccc@{}}
\toprule
$\beta$ & $\mathcal{E}$, sum rule satisfied & $\mathcal{E}$, sum rule broken \\
\midrule
0.99 & 3.74\% & 20.07\% \\
0.90 & 3.03\% & 17.65\% \\
0.70 & 1.60\% & 12.34\% \\
0.50 & 0.68\% & 7.55\% \\
0.30 & 0.23\% & 3.87\% \\
\bottomrule
\end{tabular}
\end{table}

Finally, Figure~\ref{fig:flux} shows the time evolution of the one-sided vertex fluxes $q_j(t) = -\kappa_j^2\,\partial_x u_j(x_j,t)|_{x_j=0}$, $j=1,2,3$, 
both when the sum rule holds and when it is broken. 
With our sign convention (each $q_j$ computed in bond $j$'s own local coordinate), 
the Kirchhoff condition~\eqref{eq:vertexconds} is exactly $q_1=q_2+q_3$ rather than $q_1+q_2+q_3=0$; 
this identity is confirmed numerically to machine precision ($\max_t|q_1-(q_2+q_3)|\sim10^{-17}$--$10^{-18}$) in both panels of Figure~\ref{fig:flux}, 
regardless of whether the sum rule holds.
This is consistent with the observation already made in~\cite{Matrasulov2026} for $\beta=1$: 
flux conservation at the vertex is a direct consequence of the Kirchhoff condition alone and holds irrespective of the sum rule, which instead governs how the incoming flux is apportioned between $q_2$ and $q_3$ and, more importantly, whether bond~1 itself sees any reflection (Figure~\ref{fig:snapshots}). 
The two panels of Figure~\ref{fig:flux} indeed look qualitatively similar in shape; 
the reflection caused by breaking the sum rule is visible in the temperature profile on bond~1 (Figure~\ref{fig:snapshots}) and in the transparency error $\mathcal{E}(\kappa_1)$ (Figure~\ref{fig:error}), not directly in the vertex flux balance itself.

\begin{figure}[h]
\centering
\includegraphics[width=\linewidth]{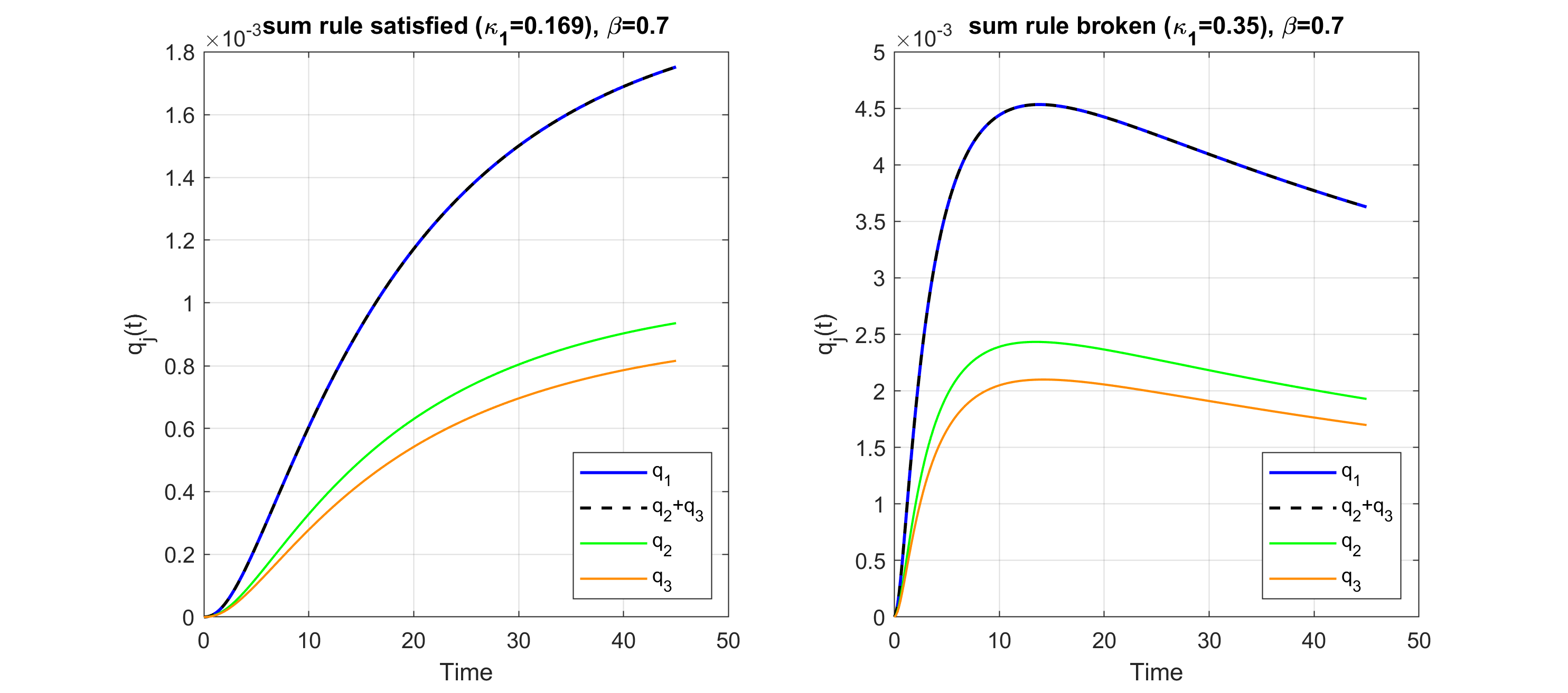}
\caption{Vertex heat fluxes $q_j(t)$, $\beta=0.7$, sum rule satisfied (left) vs.\ broken (right). The dashed curve $q_2+q_3$ overlays $q_1$ exactly in both panels, confirming Kirchhoff flux conservation independently of the sum rule.}
\label{fig:flux}
\end{figure}

\section{Summary and Discussion}\label{sec:discussion}
We have shown that thermal transparency at a branched junction survives the passage from classical to sub-diffusive transport.
The vertex operator seen by the incoming bond is an Abel-type convolution kernel of order $\beta/2$, and the diffusion-coefficient sum rule $\kappa_1=\kappa_2+\kappa_3$ that makes the junction transparent at $\beta=1$ continues to do so, unchanged, for every $\beta\in(0,1]$ : the fractional order enters both the branched and free-space operators through the same kernel and cancels out of the transparency condition. 
This was confirmed numerically with an implicit L1 scheme, both on bond~1 against an unbranched reference line and, directly, in the vertex heat flux.

Several extensions remain open. 
The sum-rule invariance generalizes to $n$-bond star graphs, but trees, loops, and graphs with multiple vertices are not treated here. 
We assumed a single fractional order shared by all bonds; bond-dependent orders $\beta_j$ break the common-kernel argument behind Proposition~\ref{prop:invariance} and are related to, but not resolved by, the edge-dependent-order model of~\cite{SobirovTuremuratova2024} and the space-time fractional metric-graph framework of~\cite{MehandirattaMehra2025}. 
Proposition~\ref{prop:wellposed} gives uniqueness for the coupled vertex problem; a matching existence theory, and a genuinely discrete (``discretize-first'') fractional vertex condition, which was suggested by 
Ehrhardt in the spirit of the 
discrete TBCs for classical parabolic systems developed in~\cite{ZisowskyEhrhardt2006} are left for future work.

\backmatter


\subsubsection*{Declarations}

\subparagraph{Author Contributions} All authors contributed equally to this work.

\subparagraph{Funding} The authors declare that there was no funding for carrying out this research.

\clearpage
\bibliography{refs}

\end{document}